\documentclass{amsart}
\pdfoutput=1

\usepackage{amssymb}
\usepackage{microtype}
\usepackage{tikz}
\usetikzlibrary{arrows}
\usepackage{booktabs}
\usepackage[pdftex,colorlinks,citecolor=black,linkcolor=black,urlcolor=black,bookmarks=false]{hyperref}

\def\arXiv#1{arXiv:\href{http://arXiv.org/abs/#1}{#1}}
\usepackage{doi}
\usepackage{subfigure}
\usepackage{bbm} 

\newtheorem{theorem}{Theorem}

\newtheorem{lemma}[theorem]{Lemma}

\newtheorem{claim}[theorem]{Claim}

\theoremstyle{definition}

\numberwithin{figure}{section}
\numberwithin{equation}{section}
\numberwithin{table}{section}

\newcommand{\R}{\mathbb{R}}
\newcommand{\C}{\mathbb{C}}

\newcommand{\E}{\mathbb{E}}
\newcommand{\prob}{\mathbb{P}}
\newcommand{\var}{ \text{var} }

\newcommand{\vol}{\operatorname{vol}}

\DeclareMathOperator{\tr}{tr}

\title{Shrinking scale equidistribution for monochromatic random waves on compact manifolds}

\author{Matthew de Courcy-Ireland}
\address{Department of Mathematics\\
Princeton University\\
Princeton NJ 08544} \email{mdc4@math.princeton.edu}

\date{February 14, 2019}

\begin{document}

\begin{abstract}
We prove equidistribution at shrinking scales for the monochromatic ensemble on a compact Riemannian manifold of any dimension. This ensemble on an arbitrary manifold takes a slowly growing spectral window 
in order to synthesize a random function.
With high probability, equidistribution takes place close to the optimal wave scale and simultaneously over the whole manifold.
The proof uses Weyl's law to approximate the two-point correlation function of the ensemble, and a Chernoff bound to deduce concentration.
\end{abstract}

\maketitle

\section{Introduction}

Consider a compact manifold $M$ together with a Riemannian metric $g$. By compactness, the spectrum of the Laplacian is a discrete sequence of eigenvalues $0 = t_0^2 \leq t_1^2 \leq t_2^2 \leq \ldots \rightarrow \infty$, possibly with multiplicity. The corresponding eigenfunctions $\phi_j : M \rightarrow \R$ satisfy
\begin{equation}
\Delta \phi_j + t_j^2 \phi_j = 0.
\end{equation}
These eigenfunctions form an orthonormal basis for $L^2(M)$, the $L^2$ space with respect to integration against the volume form of $g$. Thus one can expand functions in terms of the Laplace eigenfunctions, and a natural model for a random function on $M$ is to randomize the coefficients in such an expansion. The \emph{monochromatic ensemble} takes the specific form
\begin{equation} \label{eq:mono}
\phi(x) = \sum_{T -\eta(T) \leq t_j < T} c_j \phi_j(x)
\end{equation}
where the coefficients $c_j$ are independent, identically distributed Gaussian random variables of mean 0. The parameter $T$ is large. If the window $\eta(T)$ is short compared to $T$, then $\phi(x)$ is a stand-in for a ``random eigenfunction" with eigenvalue $T^2$. The problem with literally taking a random eigenfunction is that when an eigenvalue has multiplicity 1, the random function would simply be a deterministic function multiplied by a random scalar. 

Consider a ball $B = B_r(z)$ with center $z \in M$ whose radius $r > 0$ is allowed to vary with $T$. We can normalize so that $\int_B \phi^2$, in expectation, is close to $\vol(B)$. 
\begin{theorem} \label{thm:main}
If $rT/\log(T) \rightarrow \infty$ (or in case $\dim{M}=2, \ rT/\log(T)^2 \rightarrow \infty$) and the spectral window obeys $\eta(T)/\log(T) \rightarrow \infty$ and $\eta(T) \lesssim T^{1/2}$, then for any $\varepsilon > 0$,
\[
\prob \left\{ \sup_z \left| \frac{1}{\vol(B_r(z))} \int_{B_r(z)} | \phi |^2 - \E\left[\frac{1}{\vol(B_r(z))} \int_{B_r(z)} | \phi |^2 \right] \right| \geq \varepsilon \right\} \rightarrow 0.
\]
\end{theorem}

The \emph{wave scale} $1/T$ is the natural wavelength of an eigenfunction with Laplace eigenvalue $T^2$, also called the \emph{Planck scale} or \emph{de Broglie wavelength}. At such a fine scale, there could be a large discrepancy between $\int_B | \phi^2 |$ and $\vol(B)$. For instance, $\int_B \phi^2$ may be much larger than $\vol(B)$ if $\phi$ achieves its maximum inside $B$. The hypothesis of Theorem~\ref{thm:main} is that $r$ is large compared to the wave scale in the sense that $rT/\log(T) \rightarrow \infty$. We then conclude there is only a small deviation even in the worst case over all centers $z$. The assumption is a relatively mild one, as it allows $rT/\log(T)$ to grow arbitrarily slowly so that Theorem~\ref{thm:main} takes place almost at the wave scale.

Theorem~\ref{thm:main} follows from a more explicit bound: for any $\varepsilon > 0$, there are positive $C_{\varepsilon}$ and $c(\varepsilon)$ such that the probability of an $\varepsilon$-deviation occurring somewhere on $M$ is at most
\begin{equation} \label{eqn:main-precise}
C_{\varepsilon} T^n \exp\left(-c(\varepsilon) \big( (rT)^{-(n-1)/2} + \eta^{-1}\big)^{-1} \right).
\end{equation}
The factor $T^n$ in (\ref{eqn:main-precise}) arises from taking a union bound over roughly $T^n$ points, separated pairwise by a distance $1/T$.
The exponential factor is an upper bound for the probability of a deviation at a single point.
Under the assumption that $\eta$ and $rT$ grow faster than logarithmically, the factor $T^n$ can be absorbed into the exponential and Theorem~\ref{thm:main} follows.
We describe the union bound in more detail in Section~\ref{sec:union}.
Section~\ref{sec:chernoff} uses a Chernoff bound to estimate the probability of a deviation at a single point. 
The result is expressed in terms of the variance of the local integrals $\int_B \phi^2$, which we estimate in Lemma~\ref{lem:variance-bound}. The key input is the Local Weyl Law for Laplace eigenfunctions, in a form proved by Canzani and Hanin \cite{CH} and described in Section~\ref{sec:semi}. This is used to estimate the two-point correlation function of $\phi$, defined in Section~\ref{sec:2point}. We complete the proof of (\ref{eqn:main-precise}) in Sections~\ref{sec:var-bound} and \ref{sec:collect}.
Section~\ref{sec:conc} concludes with some further questions and a lemma that applies if the coefficients in (\ref{eq:mono}) are not necessarily Gaussian.

To have a model for random eigenfunctions, the window $\eta$ should be as small as possible, so it is not a serious restriction to assume that $\eta \lesssim T^{1/2}$ in Theorem~\ref{thm:main}. This assumption is convenient for stating simplified estimates, but the arguments below could still be implemented as long as $\eta = o(T/\log{T})$.

We mainly have in mind real-valued functions $\phi_j  : M \rightarrow \R$, but we write absolute values in Theorem~\ref{thm:main} because a similar statement holds for complex-valued functions as well. However, the complex version is not as sharp since complex eigenfunctions may equidistribute at even smaller scales than their real counterparts. For instance, on the circle $M=S^1$, $e^{iTx}$ is uniform at all scales because its modulus is identically 1, whereas $\cos(Tx)$ is limited by the wave scale $1/T$. Nevertheless, the notation below will involve complex conjugates in order to include the complex case. It would also be appropriate to take Gaussians in the complex plane if one were interested in the complex case, instead of the real coefficients $c_j$. This can be incorporated into the same proof as for the real case, since a single complex Gaussian is equivalent to two independent real Gaussians.

To provide some context for Theorem~\ref{thm:main}, consider the property of \emph{quantum unique ergodicity} (QUE). By QUE for a Riemannian manifold $M$, we mean that for any fixed measurable subset $A$ of $M$,
\begin{equation} \label{eqn:equi}
\int_A |\phi_{\lambda}|^2 d\vol \rightarrow \vol(A)
\end{equation}
for any sequence of Laplace eigenfunctions $\phi_{\lambda}$ with growing eigenvalue $\lambda \rightarrow \infty$.
There is a further question of the distribution of the microlocal lifts of $|\phi|^2 d\vol$ to phase space $S^* M$, but we confine our attention to the base space $M$.
If (\ref{eqn:equi}) holds along a full subsequence of eigenfunctions, the manifold enjoys \emph{quantum ergodicity} but may lack uniqueness of quantum limits. The quantum ergodicity theorem proved by Shnirelman \cite{Sh1, Sh2}, Colin de Verdi\`{e}re \cite{CdV}, and Zelditch \cite{Z} shows that negative curvature implies quantum ergodicity. Rudnick and Sarnak conjecture that the stronger property of QUE is true on any compact negatively curved surface \cite{RS}. This has been shown for examples of arithmetic origin in work of Lindenstrauss \cite{L1,L2}, and Bourgain-Lindenstrauss \cite{BouLi}, Jakobson \cite{J}, Holowinsky \cite{H}, and Holowinsky-Soundararajan \cite{HS}. For a general metric, work of Anantharaman \cite{A}, Anantharaman-Nonnenmacher \cite{AN}, Anantharaman-Silberman \cite{AS}, and Dyatlov-Jin \cite{DJ} places constraints on the measures that arise as quantum limits but it remains unknown whether the uniform measure is the only possibility.

From this point of view, it is of interest to randomize and see whether one at least has uniform distribution with high probability. VanderKam \cite{VdK} showed that one does have equidistribution for random spherical harmonics on the sphere, where QUE is known to fail. A more refined question is whether there is equidistribution even if the test set $A$ shrinks as the frequency grows. This scenario has been studied recently in papers of Han \cite{Han} (assuming high multiplicity), Han-Tacy \cite{HT} (with a spectral window instead of high multiplicity), Granville-Wigman \cite{GW} (on an arithmetic torus guaranteeing high multiplicity), Lester-Rudnick \cite{LR} (on higher-dimensional tori), Humphries \cite{Hum} (for non-random functions on arithmetic surfaces, with the averaging being done over the sphere center instead). 
In particular, Theorem 4.4 from Han-Tacy \cite{HT} estimates the probability that there is some point with a given deviation, much like our Theorem~\ref{thm:main} but in a different context. In \cite{HT}, instead of fluctuating near 1, $\int_M \phi^2$ is conditioned to be exactly 1. This is more natural for the quantum interpretation, but the corresponding coefficients in (\ref{eq:mono}) are no longer independent random variables, and Han-Tacy treat this with an elegant application of L\'evy's concentration of measure in high-dimensional spheres. The radius in \cite{HT} is $r = T^{-p}$ with $p$ close to $1/2$, whereas we take $r$ equal to $T^{-1}$ up to a logarithmic power.
Thus Theorem~\ref{thm:main} is closer to the wave scale, but in the easier case of a fixed $\varepsilon > 0$ instead of the shrinking deviation from \cite{HT}.

\section{Two-point function} \label{sec:2point}

A fundamental quantity governing the statistics of random functions of the form (\ref{eq:mono}) is the \emph{two-point function} of the ensemble, given by
\begin{equation} \label{eq:kernel}
K(x,x') = \sum_{T-\eta < t_j \leq T} \phi_j(x) \overline{\phi_j(x')}.
\end{equation}
At each point, $\phi(x)$ is a Gaussian of mean zero, and it is $K(x,x')$ that records the correlation of these random variables at different points on the manifold. Indeed, suppose the coefficients $c_j$ in (\ref{eq:mono}) are independent with mean 0 and variance $\sigma^2 = \E[c_j^2]$. We then have
\begin{equation}
\E[\phi(x) \overline{\phi(x')} ] = \sum_j \sum_k \phi_j(x) \overline{ \phi_k(x')} \E[c_j c_k] = \sigma^2 K(x,x').
\end{equation}

A natural normalization is to require
\begin{equation}
\E\left[ \frac{1}{\vol(M)} \int_M |\phi|^2  \right] = 1.
\end{equation}
To arrange this, the variance of the coefficients must be
\begin{equation}
\sigma^2 = \frac{\vol(M)}{\int_M K(x,x) dx } = \frac{\vol(M)}{\sum \int_M | \phi_j |^2 }.
\end{equation}
The basis functions are orthonormal in $L^2(M)$, so the denominator is just the number of eigenvalues in the interval, say $N$:
\begin{equation}
\sum_j \int_M \phi_j^2 = \#\{j \ ; \ T -\eta(T) \leq t_j  \leq  T \} = N.
\end{equation}
Thus we choose the variance of the coefficients to be
\begin{equation}
\sigma^2 = \var[c] = \frac{\vol(M)}{N} \asymp N^{-1}.
\end{equation}
 For other sets $B \subseteq M$, we then have
\[
\E \left[ \int_B |\phi |^2 \right] =\sigma^2  \int_B K(x,x) dx = \vol(B) \frac{\int_B K(x,x) dx /\vol(B) }{\int_M K(x,x) dx/\vol(M)}
\]
In the homogeneous case, $K(x,x)$ is independent of $x$ and the expectation is simply $\vol(B)$. In general, it is never very far from $\vol(B)$, as we will see from Weyl's law:
\[
\sigma^2 \int_B K(x,x) dx = \vol(B) \sigma^2 \left(\frac{N}{\vol(M)} + O(T^{n-1}) \right) = \vol(B) \left(1+ O\big( \eta^{-1} \big) \right)
\]

\section{Outline of the proof: Union bound} \label{sec:union}

To prove Theorem \ref{thm:main}, we follow the strategy of \cite{dci}. We write the random variable of interest as
\begin{equation}
X_z = \frac{1}{\vol(B_r(z))} \int_{B_r(z)} | \phi |^2.
\end{equation}
It has expectation $\E[X_z] = 1 + O(\eta^{-1})$ of order 1.
The key point is that for a monochromatic wave $\phi$ of frequency $T$, the modulus of continuity at scale $1/T$ is under control. This allows one to replace the supremum over all $z \in M$ by a maximum over roughly $T^n$ sample points, where $n = \dim(M)$.
The union bound is that for a finite number of points $z$
\begin{equation}
\prob\{ |X_z - \E X_z | > \varepsilon \ \text{for some} \ z\} \leq (\text{number of points} ) \max_z \prob\{ |X_z - \E X_z | > \varepsilon \}.
\end{equation}
For our application, the number of points is proportional to $T^n$. By the union bound, there will be only a $o(1)$ probability of there being \emph{some} point $z$ at which a deviation of $\varepsilon$ occurs, provided the probability of a deviation at any \emph{single} point $z$ is $o(T^{-n})$. Thus the union bound reduces the problem to a calculation at a single point. That calculation can be done by a Chernoff bound.

Passing to the grid brings with it another error: Conceivably the integrals around all the gridpoints are within $\varepsilon$ of their average, but nevertheless the integral around some point off the grid differs considerably. We must show that this ``off-grid" error occurs with only a low probability.

To be more precise, suppose there is a point $z$ such that
\[
|X_z - \E[X_z] | > \varepsilon.
\]
Take a grid of points $z_j$ such that every point of $M$ is within $1/T$ of a gridpoint. The number of gridpoints is thus of order $T^n$. We have
\[
\varepsilon < | X_z - X_{z_j} | + |X_{z_j} - \E[X_{z_j}] | + |\E[X_{z_j}] - \E[X_z] |
\]
Thus one of the three terms must be greater than $\varepsilon/3$. The difference of expected values is non-random and small: Both are $1 + O(\eta^{-1})$, so their difference is $O(\eta^{-1})$. Eventually, this will not be greater than $\varepsilon/3$ since we assume $\eta(T) \rightarrow \infty$. Alternatively, note that
\begin{align*}
| \E[X_{z_{j}}] - \E[X_z] | &= \sigma^2 \left| \frac{1}{\vol(B_r(z))} \int_{B_r(z)} K(x,x) dx - \frac{1}{\vol(B_r(z_j)} \int_{B_r(z_j)} K(x,x) dx \right| \\
&\lesssim \frac{ \vol(B_r(z) \Delta B_r(z) ) }{\vol(B_r)} \\
\end{align*}
To bound the volume of the symmetric difference, we have the following claim.
\begin{claim}
If $B_r(z)$ and $B_r(z')$ are balls of radius $r \rightarrow 0$ centered at points $z, z'$ separated by less than $r$ in a Riemannian manifold of dimension $n$,
\begin{equation}
\vol(B_r(z) \Delta B_r(z') ) \lesssim r^{n-1} d(z,z') .
\end{equation}
\end{claim}
\begin{proof}
Indeed, for small radii $r$, we can compare to Euclidean balls or simply to a Euclidean box with $n-1$ sidelengths of order $r$ and a remaining side of order $s = d(z,z')$. The bound $r^{n-1}s$ holds for larger separations as well, but becomes worse than the easier bound 
\[
\vol(B \Delta B') \lesssim \vol(B)+\vol(B') \lesssim r^n.
\]
\end{proof}
With a separation of less than $1/T$ between $z$ and $z_j$, we therefore have
\[
| \E[X_{z_j}] - \E[X_z] | \lesssim \frac{r^{n-1} T^{-1} }{r^n} = \frac{1}{rT}.
\]
Assuming $rT \rightarrow \infty$, this term will be less than $\varepsilon/3$. Thus the difference of expected values will eventually be less than $\varepsilon/3$ whether we assume $\eta \rightarrow \infty$ or $rT \rightarrow \infty$ (and later, we will assume that both of them diverge faster than logarithmically). In the case of an $\varepsilon$-difference of $\int_B |\phi|^2$ from its mean, it is one of the other two terms $|X_z - X_{z_j}|$ or $|X_{z_j} - \E[X_{z_j}]|$ that must be greater than $\varepsilon/3$ (and in fact, almost greater than $\varepsilon/2$ once $rT$ and $\eta$ are large enough). 

Suppose it is the integrals around $z$ versus $z' = z_j$ that differ by more than $\varepsilon/3$. We have
\[
\left| \int_B | \phi |^2 - \int_{B'} | \phi |^2 \right| \lesssim \int_{B \Delta B'} |\phi|^2 \lesssim \vol(B \Delta B') \| \phi \|_{\infty}^2.
\]
Since $d(z,z_j) < 1/T$, the same volume bound as above gives
\[
\frac{\varepsilon}{3} \lesssim r^{-n} \left( r^{n-1} T^{-1} \| \phi \|_{\infty}^2 \right).
\]
That is,
\[
\| \phi \|_{\infty} \gtrsim \sqrt{\varepsilon rT}.
\]
To control the probability of $\phi$ having such a large maximum, we use another union bound. 
More precise estimates of $\| \phi \|_{\infty}$ have been given by Burq-Lebeau \cite{BL} and Canzani-Hanin \cite{CH}, but we include the following sketch to keep the present argument self-contained.
Again, take a grid of roughly $T^n$ points. Either there is a gridpoint $w_j$ at which $|\phi(w_j) | \geq C \sqrt{\varepsilon rT}$ or else there are two points separated by only $1/T$ at which the values of $\phi$ differ by at least $C \sqrt{\varepsilon rT}$. The latter is very unlikely because $1/T$ is the wave scale for $\phi$. Whereas the values $\phi(w)$ are Gaussian with unit variance, the derivatives of $\phi$ are Gaussian with variance $T^2$, so a difference of $C\sqrt{\varepsilon rT}$ between points separated by only $1/T$ would require $\phi$ to have some directional derivative more than $\sqrt{\varepsilon rT}$ standard deviations above its mean. This occurs with probability less than $\exp(-c \varepsilon rT)$. Likewise, having $|\phi(w_j) | \geq C \sqrt{\varepsilon rT}$ requires a Gaussian to be more than $\sqrt{\varepsilon rT}$ standard deviaions above its mean. From the union bound,
\[
\prob( \| \phi \|_{\infty} \geq c \sqrt{\varepsilon rT} ) \lesssim T^n \exp(-c' \varepsilon rT)
\]
which is negligible as long as $rT/\log(T) \rightarrow \infty$. Thus we can move to the final case: The probability that an integral around any single point shows a deviation of more than $\varepsilon/3$.

\section{Chernoff bound} \label{sec:chernoff}

Each variable $X_z$ is a quadratic form in the coefficients $c_j$. Writing $B = B_r(z)$, we have
\begin{equation}
X_z = \frac{1}{\vol(B)} \int_B |\phi|^2 = \sum_j \sum_k c_j c_k \frac{1}{\vol(B)} \int_B \phi_j \overline{\phi_k}.
\end{equation}
We scale by the variance to write $c_j = \sigma \mathfrak{z}_j$, where $\mathfrak{z}_j$ is a standard Gaussian of mean 0 and variance 1. Thus
\begin{equation}
X_z = \mathfrak{z}^T A \mathfrak{z}
\end{equation}
where the matrix $A$ has entries
\begin{equation}
A_{jk} = \frac{\sigma^2}{\vol(B)} \int_B \phi_j \overline{\phi_k}.
\end{equation}
Note that this matrix depends on $z$, as well as $r$ and $T$, but we have suppressed this in the notation.
Since $A$ is a symmetric matrix, or Hermitian if we prefer to start from complex-valued eigenfunctions $\phi_j$, we may diagonalize to write $A = U^T D U$ where $U$ is orthogonal (or unitary, in the complex case) and $D$ is diagonal with entries, say, $\lambda_j$.
In eigencoordinates, the random variable $X_z$ becomes
\begin{equation}
X_z = \mathfrak{z}^T A\mathfrak{z} = (U\mathfrak{z})^T D (U\mathfrak{z}) = \sum_j \lambda_j y_j^2
\end{equation}
where $y = U\mathfrak{z}$ is again a standard Gaussian vector. 

Evaluating a Gaussian integral, it follows that the moment generating function of a quadratic form $\mathfrak{z}^T A\mathfrak{z}$ in standard Gaussians $\mathfrak{z} = (\mathfrak{z}_1,\ldots, \mathfrak{z}_N)$ is
\begin{equation} \label{eq:mgf}
g(s) = \E\left[ e^{s\mathfrak{z}^TA\mathfrak{z}} \right] = \prod_{j=1}^{N} (1-2s\lambda_j)^{-1/2}
\end{equation}
where $\lambda_j$ are the eigenvalues of $A$. In the complex case, each factor effectively occurs twice because of the real and imaginary parts of $y_j$, leading to $(1-2s\lambda_j)^{-1}$ instead of $(1-2s\lambda_j)^{-1/2}$. 
One has convergence in (\ref{eq:mgf}) as long as $1-2s\lambda_j > 0$ for all $j$, so $s$ must be small enough. Specifically, $g(s)$ is defined for $s < 1/(2\lambda_{\max})$, where $\lambda_{\max}$ is the largest eigenvalue of $A$. 

Estimates for $g(s)$ allow us to execute a Chernoff bound on the tail probability.
For any $s > 0$, $X > \E[X] + \varepsilon$ if and only if $e^{sX} > e^{s\E[X] + s\varepsilon}$, so by Markov's inequality
\begin{equation}
\prob \{X > \E[X] + \varepsilon \} \leq g(s) e^{-s\E[X] - s\varepsilon} = \exp \left(-s\varepsilon - s\E[X] + \log{g(s)} \right).
\end{equation}
In the case at hand, where $X = \mathfrak{z}^T A \mathfrak{z}$, we have
\begin{equation}
-s\varepsilon - s\E[X] + \log{g(s)} = -s\varepsilon - s\E[X] + \frac{1}{2} \sum_{j} -\log(1-2s\lambda_j).
\end{equation}
Expanding the logarithm in a power series (provided $2s\lambda_{\max} < 1$), we have
\[
\frac{1}{2} \sum_j -\log(1- 2s\lambda_j) = \sum_{p=1}^{\infty} \frac{1}{2p} \sum_j (2s\lambda_j)^p.
\]
The term $p=1$ contributes $s\sum_j \lambda_j = s\E[X]$. This cancels the expected value above so that
\begin{align*}
-s\varepsilon - s\E[X] + \log{g(s)} &= -s\varepsilon + \sum_{p \geq 2} \frac{1}{2p} \sum_j (2s\lambda_j)^p \\
&= -s\varepsilon + s^2 \sum_j \lambda_j^2 + \sum_{p \geq 3} \frac{1}{2p} \sum_j (2s\lambda_j)^p.
\end{align*}
We would like to minimize the sum of the first two terms by choosing
\begin{equation}
s_? = \frac{\varepsilon}{2 \sum \lambda_j^2}
\end{equation}
but it is not clear whether $2s_? \lambda_{\max} < 1$, that is, whether $g(s_?)$ is defined. We would need to know that 
\[
\lambda_{\max} < \frac{1}{\varepsilon} \sum_j \lambda_j^2
\]
at least for sufficiently small $\varepsilon$. 
In the case of the manifold $S^2$ with its usual round metric, we were able to show in \cite{dci} that $\lambda_{\max}$ and $\sum \lambda_j^2$ are of the same order of magnitude, so that this holds once $\varepsilon$ is small enough. 
Here, we choose a different $s$ to guarantee that $2s\lambda_{\max} < 1$, namely
\begin{equation}
s = c \left( \sum_j \lambda_j^2 \right)^{-1/2}
\end{equation}
where $c < 1/2$. Note that $\lambda_{\max} \leq \sqrt{ \sum \lambda_j^2 }$, so that this is a valid choice of $s$.

\begin{claim}
For this choice $s = c/\sqrt{\sum \lambda_j^2}$, where $0 < c < 1/2$, we have
\begin{equation}
\log{g(s)} - s\E[X] \leq A s^2 \sum_j \lambda_j^2
\end{equation}
where $A$ can be taken as $2c^2/(1-2c)^2$.
\end{claim}
\begin{proof}
Indeed, this follows from Taylor's theorem. For a twice differentiable function $f$, we have
\[
f(x) = f(a) + f'(a) (x-a) + \int_a^x f''(t) (x-t) dt
\]
Applied to the function $f(x) = -\log(1-x)$, this gives
\[
-\log(1-x) = x + \int_0^x \frac{1}{(1-t)^2} (x-t) dt.
\]
In particular, for $x \leq a$ we have
\[
-\log(1-x) - x \leq x^2 (1 - a)^{-2}
\]
so we may take $A = (1-a)^{-2}$ to have a bound valid for all $x$ up to $a$. We take $x = 2s\lambda_j$ where $s = c(\sum \lambda_j^2)^{-1/2}$ with $0 < c < 1/2$. These values of $x$ are at most
\[
x = 2s\lambda_j \leq 2c \frac{\lambda_{\max}}{(\sum \lambda_j^2 )^{1/2} } \leq 2c.
\]
Taylor's theorem then gives
\[
-\log(1-2s\lambda_j) - 2s\lambda_j \leq (1-2c)^{-2} 4s^2 \lambda_j^2 = \frac{4c^2}{(1-2c)^2} \lambda_j^2 / \sum_i \lambda_i^2 
\]
Summing over $j$ and dividing by 2, we get
\[
-\log{g(s)} - s \sum_j \lambda_j \leq \frac{2c^2}{(1-2c)^2}
\]
Hence, noting again that $\sum_j \lambda_j = \E[X]$, we have proved the claim.
\end{proof}

With this estimate in hand, we can bound the tail probability as follows:
\begin{equation}
\prob\{ X > \E[X] + \varepsilon \} \leq e^{2c^2/(1-2c)^2} \exp \left(-c \varepsilon \big( \sum_j \lambda_j^2 \big)^{-1/2} \right)
\end{equation}

The lower tail, where $X < \E[X] - \varepsilon$, is slightly different but can be treated by the same method. We have $X < \E[X] - \varepsilon$ if and only if $-X > \E[-X] + \varepsilon$, so we can apply the argument above with $-X$ in place of $X$. Instead of $g(s)$, the relevant function for the Chernoff bound is
\begin{equation} \label{eq:g-}
g_-(s) = \E\left[ e^{-sX} \right] = \prod_j (1+2s\lambda_j)^{-1/2}.
\end{equation}
This function $g_-(s)$ is defined for all $s \geq 0$ whereas $g(s)$ is defined only for sufficiently small $s$. The Chernoff bound is
\begin{equation}
\prob\{ -X > \E[-X] + \varepsilon \} \leq g_-(s) e^{s\E[X]}  e^{-s\varepsilon} .
\end{equation}
We have $-\log(1+x) \leq -x + x^2/2$ for all $x \geq 0$, so that
\[
\log{g_-(s)} + s\E[X] \leq \frac{1}{4} \sum_j (2s \lambda_j)^2 \leq c^2
\]
where we choose $s = c \big( \sum \lambda_j^2 \big)^{-1/2}$ as above. This shows that the lower tail probability obeys the same bound as the upper tail probability, namely
\begin{equation}
\prob\{ -X > \E[-X] + \varepsilon \} \leq e^{c^2} \exp \left(-c \varepsilon \big( \sum \lambda_j^2 \big)^{-1/2} \right).
\end{equation}
In fact, since $g_-(s)$ is defined for all $s$, we could simply choose $s = s_?$ to get an even better bound. This doesn't help us though, since we control both upper and lower tail together by the sum of their respective bounds:
\begin{align*}
\prob\{ |X - \E[X] | > \varepsilon \} &\leq \prob\{ -X > \E[-X] + \varepsilon \} + \prob\{ X > \E[X] + \varepsilon \} \\
&\leq \left( e^{2c^2/(1-2c)^2}+ e^{c^2} \right)   \exp\left(-c \varepsilon \left( \sum \lambda_j^2 \right)^{-1/2} \right)
\end{align*}
for any $c<1/2$.

In order to take advantage of this, we need an estimate on the second moment $\sum \lambda_j^2$. 
\begin{lemma} \label{lem:variance-bound}
\begin{equation}
\sum_j \lambda_j^2 \lesssim \left( (rT)^{-(n-1)/2} + \eta^{-1} \right)^2.
\end{equation}
\end{lemma}
We will prove the lemma using estimates for the two-point function $K(x,x')$. We have
\[
\sum_j \lambda_j^2 = \tr(A^2).
\]
The trace $\tr(A^2)$, and also the trace of any power of $A$, can be expressed in terms of $K(x,x')$ as follows.

Recall that
\[
K(x,x') = \sum_j \phi_j(x)\overline{\phi_j(x')}.
\]
Since the $(j,k)$-entry of $A$ is
\begin{equation*}
A_{jk} = \frac{\sigma^2}{\text{vol}(B)} \int_{B} \phi_j \overline{\phi_k},
\end{equation*}
the entries of $A^p$ are
\begin{equation*}
A_{jk}^{(p)} = \frac{\sigma^{2p}}{\text{vol}(B)^p}\sum_{k_1} \cdots \sum_{k_{p-1}} \int_{B} \phi_j \overline{\phi_{k_1}} \int_{B} \phi_{k_1} \overline{\phi_{k_2}} \ldots \int_{B} \phi_{k_{p-2}} \overline{\phi_{k_{p-1}}} \int_{B} \phi_{k_{p-1}} \overline{\phi_k}.
\end{equation*}
When we sum the diagonal entries, we get
\begin{align*}
\tr(A^p) &= \sum_j A_{jj}^{(p)} \\
&= \frac{\sigma^{2p}}{\vol(B)^{p}} \sum_j \sum_{k_1} \cdots \sum_{k_{p-1}} \int_{B} \phi_j \overline{\phi_{k_1}} \int_{B} \phi_{k_1} \overline{\phi_{k_2}} \ldots \int_{B} \phi_{k_{p-2}} \overline{\phi_{k_{p-1}}} \int_{B} \phi_{k_{p-1}} \overline{\phi_j} 
\end{align*}
We can equally well express this product of integrals as one multiple integral:
\begin{align*}
&\tr(A^p) = \vol(B)^{-p} \int_B dx_1  \cdots \int_B dx_p  \\
&\sum_j \sum_{k_1} \cdots \sum_{k_{p-1}} \phi_j(x_1) \overline{\phi_{k_1}(x_1)} \phi_{k_1}(x_2)\overline{\phi_{k_2}(x_2)} \ldots \phi_{k_{p-2}}(x_{p-1}) \overline{\phi_{k_{p-1}}(x_{p-1})} \phi_{k_{p-1}}(x_{p}) \overline{\phi_j(x_{p})}
\end{align*}
The integrand factors:
\begin{align*}
&\sum_j \sum_{k_1} \cdots \sum_{k_{p-1}} \phi_j(x_1) \overline{\phi_{k_1}(x_1)} \phi_{k_1}(x_2)\overline{\phi_{k_2}(x_2)} \ldots \phi_{k_{p-2}}(x_{p-1}) \overline{\phi_{k_{p-1}}(x_{p-1})} \phi_{k_{p-1}}(x_{p}) \overline{\phi_j(x_{p})} \\
&= \sum_j \phi_j(x_1) \overline{\phi_j(x_{p})} \sum_{k_1} \overline{\phi_{k_1}(x_1)} \phi_{k_1}(x_2) \cdots \sum_{k_{p-1}} \overline{\phi_{k_{p-1}}(x_{p-1})}\phi_{k_{p-1}}(x_{p}) \\
&= K(x_1,x_{p}) K(x_2,x_1) \cdots K(x_{p}, x_{p-1}) 
\end{align*}
We summarize this as follows:
\begin{lemma} \label{lem:trap}
If $A$ is the matrix with entries
\begin{equation}
A_{jk} = \frac{\sigma^{2p}}{\vol(B)} \int_B \phi_j \overline{\phi_k}
\end{equation}
and $K$ is the kernel given by
\begin{equation}
K(x,x') = \sum_j \phi_j(x) \overline{\phi_j(x')}
\end{equation}
then
\begin{equation}
\tr(A^p) = \frac{\sigma^{2p}}{\vol(B)^p} \int_B \cdots \int_B \prod_{j=1}^p K(x_j, x_{j-1}) \ dx_1 \ldots dx_{p}
\end{equation}
with the indices interpreted cyclically so that $x_0$ means $x_p$.
\end{lemma}

In particular, with $p=2$, we have
\begin{equation} \label{eqn:tra2}
\tr(A^2) = \frac{\sigma^4}{\vol(B)^2} \int_B dx_1 \int_B dx_2 |K(x_1, x_2)|^2.
\end{equation}

\section{Input from semiclassics} \label{sec:semi}

To prove the variance estimate in Lemma~\ref{lem:variance-bound} , we need to know the size of $K(x,x')$. Here is the basic estimate: 
\begin{claim} \label{claim:simple}
On a compact manifold of dimension $n$, with spectral kernel
\[
K(x,x') = \sum_{T-\eta < t_j \leq T} \phi_j(x) \overline{\phi_j(x')}
\]
defined over a window $\eta(T) \rightarrow \infty$ growing arbitrarily slowly and such that 
\[
\eta(T) \lesssim T^{1/2},
\]
we have
\begin{equation}
K(x,x') \lesssim T^{n-1} \eta(T)
\end{equation}
for all $x, x'$ and an improved bound for well-separated pairs:
\begin{equation} \label{eqn:upp}
K(x,x') \lesssim T^{n-1}\eta \left( (Td(x,x'))^{-(n-1)/2} + \eta^{-1} \right)
\end{equation}
improving on the trivial bound once $d(x,x') > 1/T$.
\end{claim}

For $d(x,y) \lesssim 1/T$, the basis for claim \ref{claim:simple} is H\"{o}rmander's Theorem 4.4 from \cite{H}.
This in turn is based on Lax's parametrix for the wave equation, constructed in \cite{L}.
Using the wave equation in this way may break down when $Td(x,y)$ is unbounded.
For larger distances we instead appeal to the results of Canzani-Hanin \cite{CH2}.
Their Theorem 2 improves the $O(T^{n-1})$ error term in H\"ormander's estimate for $K(x,y)$ to $o(T^{n-1})$, assuming $x,y$ are in a ball $B_r(z)$ of radius $r \rightarrow 0$ arbitrarily slowly around some non-self-focal point $z$.
Without the assumption on $z$, one cannot conclude the remainder is $o(T^{n-1})$ since the sphere is a counterexample, but the method of \cite{CH2} still gives
\begin{equation} \label{eqn:canhan}
\sum_{t_j \leq T} \phi_j(x)\phi_j(y) = \frac{T^n}{(2\pi)^n} \int_{|\xi|_{g_y}<1} e^{i T \langle \exp_y^{-1}(x), \xi \rangle_{g_y} } \frac{d\xi}{\sqrt{|g_y|}} + O\left( T^{n-1} \right)
\end{equation}
where the error term is uniform over pairs $(x,y)$ with $d(x,y) < r$. In this notation, $g_y$ and $|*|_{g_y}$ are the length and inner product on the tangent space at $y$ defined by the metric $g$, $\sqrt{|g_y|}$ is the volume form, and $\exp_y$ is the exponential map. Note that $\exp_y^{-1}(x)$ is well defined for $d(x,y)$ sufficiently small (less than the injectivity radius of $M$).

Using polar coordinates at $y$, with $\omega = \exp_y^{-1}(x)$ and $\xi = s \alpha$, the difference between the main terms for $T$ and $T-\eta$ is
\begin{align*}
&\left( \frac{T}{2\pi} \right)^n \int_{s=1-\eta/T}^1 \int_{S^{n-1}} e^{iTd(x,y) \omega \cdot \alpha} s^{n-1} ds d\alpha \\
&= \left(\frac{T}{2\pi}\right)^n \frac{1 - (1 - \eta/T)^n}{n} \int_{S^{n-1}} e^{iTd(x,y) \alpha \cdot \omega} d\alpha.
\end{align*}
The integral over $S^{n-1}$ gives the Bessel function 
\[
J(T(dx,y)) = J_{n/2-1}(Td(x,y))/(Td(x,y)^{n/2-1},
\]
up to a normalizing factor depending only on $n$. This is a bounded function that begins to oscillate when $Td(x,y)$ reaches the first zeros of $J_{n/2-1}$, and decays as a power $(Td(x,y)^{-n/2+1/2}$ as $Td(x,y) \rightarrow \infty$. 
We have
\begin{equation} \label{eqn:binom}
\frac{1 - (1-\eta/T)^n}{n} = \frac{\eta}{T} + O\left( \frac{\eta}{T} \right)^2
\end{equation}
by the binomial expansion. 
This implies
\begin{equation}
K(x,y) = cT^{n-1} \eta(T) \left( \frac{J_{n/2-1}(Td(x,x')) }{(Td(x,x'))^{n/2-1} }  + O\left( \eta^{-1} + \eta T^{-1} \right) \right)
\end{equation}
for some constant $c = c_n > 0$. Note that the $\eta^{-1}$ in the error corresponds to the remainder in Weyl's law whereas $\eta T^{-1}$ is from truncating the binomial expansion in (\ref{eqn:binom}). They are equal when $\eta = T^{1/2}$.

If $d(x,y) \lesssim 1/T$, we simply use the fact that $J$ is bounded to obtain the trivial bound
\[
K(x,y) \lesssim T^{n-1} \eta.
\]
This is useful for nearby pairs $(x,y)$, but for $d(x,y) \gtrsim 1/T$ it is better to input the fact that $J(u) \lesssim u^{-n/2+1/2}$ to obtain 
\[
K(x,y) \lesssim \eta T^{n-1} \left( (Td(x,y)^{-n/2+1/2} + O\big( \eta^{-1} + \eta T^{-1} \big) \right).
\]
We have assumed $\eta \lesssim T^{1/2}$ so that $\eta T^{-1}$ can be absorbed into the error $\eta^{-1}$.
This gives (\ref{eqn:upp}).
\qed

We have assumed that $\eta \lesssim T^{1/2}$ for convenience, and indeed what we have in mind is that $\eta$ is a power of $\log{T}$. If one did want to allow larger $\eta$, the error in (\ref{eqn:upp}) would become $\eta T^{-1}$ instead of $\eta^{-1}$. For the arguments in Section~\ref{sec:collect} below to go through, one would then need to assume $\eta = o(T/\log{T})$. 

\section{Upper bound on the variance} \label{sec:var-bound}

By the triangle inequality, $d(x,x') \leq d(x,z) + d(z,x') < 2r$. Since the integrand is nonnegative, we can bound the inner integral in (\ref{eqn:tra2}) by
\begin{equation}
\int_{B_r(z)}| K(x,x')|^2 dx \leq \int_{B_{2r}(x')} |K(x,x')|^2 dx.
\end{equation}
Having moved the center to $x'$, we introduce polar coordinates $(\rho, \omega)$ where the radial coordinate $\rho = d(x,x')$ ranges from 0 to $2r$. 
The volume form is given approximately by its Euclidean counterpart:
\begin{equation}
d\vol(x) = (1 + O(\rho^2) ) \rho^{n-1} d\rho d\omega.
\end{equation}
Indeed, the volume form is obtained from the metric $g$ by $\sqrt{\det(g)}$ and we have the expansion
\begin{equation*}
\sqrt{\det(g)} = 1 - \frac{1}{6} \text{Ric}_{kl}(x') x^k x^l + O(|x|^3) = 1 + O(\rho^{2} ).
\end{equation*}

We integrate the estimate (\ref{eqn:upp}) from section \ref{sec:semi}, namely
\begin{align*}
K(x,x') &\lesssim T^{n-1} \eta \left( (T\rho)^{-(n-1)/2} + \eta^{-1} \right).
\end{align*}
This diverges as $\rho \rightarrow 0$, since we would be better off using the trivial bound for $\rho < 1/T$, but the singularity is integrable.
We obtain
\begin{align*}
\int_{B_r(z)} K(x,x')^2 dx' &\lesssim (T^{n-1}\eta)^2 \int_0^{2r} \left( (T\rho)^{-(n-1)/2} + \eta^{-1} \right)^2 \rho^{n-1} d\rho \\
&\lesssim T^{2n-2} \eta^2 r^n \left( (rT)^{-(n-1)} + \eta^{-1} (rT)^{-(n-1)/2} + \eta^{-2} \right)
\end{align*}
Integrating over $x$ and noting that $\vol(B_r) \asymp r^n$, we obtain
\[
\int_B \int_B K(x,x')^2 dx' dx \lesssim \vol(B)^2 \big( T^{n-1} \eta \big)^2 \left( (rT)^{-(n-1)} + \eta^{-1} (rT)^{-(n-1)/2} + \eta^{-2} \right)
\]
as claimed in Lemma ~\ref{lem:variance-bound}.
This improves on what one would get by replacing $K$ with its maximum, namely
\[
\int_B \int_B K(x,x')^2 dx' dx \lesssim \vol(B)^2 \big( T^{n-1} \eta \big)^2
\]
Recall that we have normalized to have Gaussian coefficients of variance proportional to $T^{n-1} \eta$. Thus this factor $(T^{n-1}\eta)^2$ will cancel, leaving
\[
\var\left[ \frac{1}{\vol(B)} \int_B | \phi |^2 \right] = \frac{\sigma^4}{\vol(B)^2} \int_B \int_B K^2 \lesssim (rT)^{-(n-1)} + \eta^{-1} (rT)^{-(n-1)/2} + \eta^{-2}
\]
This vanishes as $rT \rightarrow \infty$ and $\eta \rightarrow \infty$, whereas the trivial bound would only show the variance is bounded.

\section{Collecting the bounds and proving Theorem~\ref{thm:main}} \label{sec:collect}

From the union bound, we had
\begin{align*}
\prob\left( \exists z \ |X_z - \E[X_z] | > \varepsilon \right) &\lesssim \prob ( \| \phi \|_{\infty} > c_1 \sqrt{\varepsilon rT} ) + T^n \prob( |X_z - \E[X_z ] | > \varepsilon /3 ) \\
&\lesssim T^n \left(\exp(-c_3 \varepsilon rT) + \prob( |X_z - \E[X_z ] | > \varepsilon /3 ) \right)
\end{align*}
From the Chernoff bound,
\begin{equation}
\prob( |X_z - \E[X_z ] | > \varepsilon /3 ) \lesssim \exp\left(-c_4 \varepsilon \left( \sum \lambda_j^2 \right)^{-1/2} \right)
\end{equation}
From the variance formula,
\begin{align*}
\sum \lambda_j^2 &\lesssim (rT)^{-(n-1)} + \eta^{-1} (rT)^{-(n-1)/2} + \eta^{-2} \\
&\lesssim \left( (rT)^{-(n-1)/2} + \eta^{-1} \right)^2
\end{align*}
Therefore
\begin{equation}
\prob( |X_z - \E[X_z ] | > \varepsilon /3 ) \lesssim \exp\left( -c_5 \varepsilon \left( (rT)^{-(n-1)/2} + \eta^{-1} \right)^{-1} \right)
\end{equation}
We already assumed $rT/\log(T) \rightarrow \infty$ so that $T^n \exp(-c_3 \varepsilon rT) \rightarrow 0$ no matter how small is the given $\varepsilon$, which controls the probability of an ``off-grid" deviation. To control the ``on-grid" deviation, we must further assume that
\[
\left(  \big( rT)^{-(n-1)/2} + \eta^{-1} \right)^{-1} /\log(T) \rightarrow \infty.
\]
This guarantees that, again, the factor of $T^n$ can be absorbed. Equivalently, we need 
\[
\left( (rT)^{-(n-1)/2} + \eta^{-1} \right)\log(T) \rightarrow 0,
\]
that is, both $(rT)^{-(n-1)/2} \log(T) \rightarrow 0$ and $\eta^{-1}\log(T) \rightarrow 0$. For $n \geq 3$, the first of these is already implied by the assumption $rT/\log(T) \rightarrow \infty$. 
If $n = 2$, then we instead assume $(rT)/\log(T)^2 \rightarrow \infty$. 
Thus the requirements amount to both $rT$ and $\eta(T)$ being asymptotically larger than $\log(T)$:
\begin{align*}
\frac{rT}{\log(T)} &\rightarrow \infty, \quad (\text{or} \ rT/\log(T)^2 \rightarrow \infty \ \text{if} \ n=2) \\
\frac{\eta(T)}{\log(T)} &\rightarrow \infty
\end{align*}
These are the hypotheses of Theorem~\ref{thm:main}, and the proof is complete. 
Moreover, we have proved the rate of convergence for Theorem~\ref{thm:main} claimed in (\ref{eqn:main-precise}): for any $\varepsilon > 0$, there are positive $C_{\varepsilon}$ and $c(\varepsilon)$ such that
\begin{equation}
\prob(\exists z |X_z - \E[X_z]| > \varepsilon) \leq C_{\varepsilon} T^n \exp\left(-c(\varepsilon) \big( (rT)^{-(n-1)/2} + \eta^{-1}\big)^{-1} \right).
\end{equation}

\section{Conclusion} \label{sec:conc}

The proof we have given relies on a union bound, ignoring the interesting question of how integrals $\int_B |\phi|^2$ and $\int_{B'} |\phi|^2$ over different sets are correlated. One might also wonder about other ensembles of random functions, for instance band-limited functions with a window $\eta(T)$ proportional to $T$ instead of $o(T)$, or where the distribution of the coefficients is not Gaussian. One could study other sets $B$, not necessarily balls, either with diameter shrinking like the $r$ in our setup, or volume shrinking like $r^n$. The lifts of $|\phi|^2 d\vol$ to $S^* M$ are another interesting class of random measures. Regarding more general coefficients, we note the article \cite{HW} of Hanson-Wright on concentration for quadratic forms in independent random variables.

As a first step addressing two of these further directions, here is an exact covariance formula. The covariance between two of our integrals takes a similar form to the variance of a single one. In \cite{dci}, we did this calculation on the sphere. This was an algebraic calculation valid in more general circumstances, as we now indicate. This proof applies to non-Gaussian distributions of the coefficients, as long as the first four moments are the same as for a Gaussian, whereas the proof by differentiating the moment generating function is specific to Gaussians.
Without the assumption on the fourth moment, there is a more complicated formula involving $\sum_j \phi_j(x)^2\phi_j(y)^2$ in addition to the kernel $\sum_j \phi_j(x) \phi_j(y)$.

\begin{lemma} \label{lem:variance}
Suppose $c_j$ are independent random variables with first and third moments $0$, variance $\sigma^2$, and fourth moment $3\sigma^4$. Suppose $\phi_j : M \rightarrow \C$ are functions on some measure space $M$ (assumed $\sigma$-finite for purposes of Fubini's theorem) and $\phi = \sum_j c_j \phi_j$ is the corresponding random function. Then for any measurable subsets $B \subseteq M$, $B' \subseteq M$,
\begin{equation} \label{eq:variance-formula}
{\rm cov} \left[\int_B | \phi |^2  , \int_{B'} |\phi |^2 \right] = 2\sigma^4 \int_{B} \int_{B'} K(x,x')^2 dx dx'
\end{equation}
where $K(x,x') = \sum_j \phi_j(x)\overline{\phi_j(x')}$.
If the fourth moment $\E[c^4]$ does not necessarily equal $3\sigma^4$, then the covariance is given by
\begin{align*}
{\rm cov}\left[\int_B |\phi|^2, \int_{B'} |\phi|^2 \right] = &2\sigma^4 \int_B \int_{B'} K(x,x')^2 dxdx' \\
&+ \big(\E[c^4] - 3\sigma^4 \big) \int_B \int_{B'} \sum_j \phi_j(x)^2 \phi_j(x')^2 dxdx'.
\end{align*}
\end{lemma}

\begin{proof}
We compute the covariance $\E[ \int_B |\phi |^2 \int_{B'} |\phi|^2] - \E[\int_B |\phi |^2 ]  \E[ \int_{B'} |\phi |^2]$ by expanding $| \phi |^2$ and using linearity of expectation to exchange $\E$ with the sums and integrals. For the expectation of the product, we have
\begin{equation*}
\E \left[\int_B \phi^2 \int_{B'} \phi^2 \right] = \int_B \int_{B'} \sum_i \sum_j \sum_k \sum_l \phi_i(x)\overline{\phi_j(x)} \phi_k(x') \overline{\phi_l(x')} \E[c_i c_j c_k c_l] dx dx'.
\end{equation*}
Since the coefficients are independent and have mean 0, the expectation $\E[c_i c_j c_k c_l]$ is $3\sigma^4$ if all indices $i, j, k,$ and $l$ are equal, $\sigma^4$ if they are equal in pairs, and $0$ in all other cases. In light of the different cases $i=j \neq k = l$, $i = k \neq j = l$, or $i = l \neq j = k$, it follows that
\begin{align*}
&\E \left[\int_B |\phi |^2 \int_{B'}| \phi |^2 \right] \\
&= \sigma^4 \left(3 \sum_i |\phi_i(x) |^2 |\phi_i(x')|^2 + \sum_{i \neq k} |\phi_i(x)|^2 |\phi_k(x')|^2 + 2 \sum_{i \neq j} \phi_i(x) \overline{\phi_i(x')} \phi_j(x')\overline{\phi_j(x')} \right)
\end{align*}
The factor of 3 means that the first term exactly supplies the missing diagonal terms $i=k$, $i=j$, and $i=l$ (which we have merged with $i=k$, the two cases giving the same contribution) in the three other sums. The completed sums then factor, so that
\begin{align*}
\E\left[  |\phi(x)|^2  |\phi(x')|^2 \right] &= \sigma^4 \left( \sum_i |\phi_i(x)|^2 \sum_k |\phi_k(x')|^2 + 2\sum_i \phi_i(x) \overline{\phi_i(x')} \sum_j \phi_j(x) \overline{\phi_j(x')} \right) \\
&= \sigma^4 \left( K(x,x)K(x',x') + 2K(x,x')^2 \right)
\end{align*}
For the product of the expectations, we have
\begin{equation*}
\E\left[ \int_B |\phi|^2 \right] = \sum_i \sum_j \E[c_i c_j] \int_B \phi_i \overline{\phi_j} = \sigma^2 \int_B K(x,x) dx
\end{equation*}
by independence of the coefficients. Thus subtraction gives
\begin{align*}
&{\rm cov} \left[\int_B | \phi |^2  , \int_{B'} |\phi |^2 \right]  = \E\left[ \int_B |\phi |^2 \int_{B'} |\phi|^2 \right] - \E\left[\int_B |\phi |^2 \right]  \E\left[ \int_{B'} |\phi |^2\right] \\
&= \sigma^4 \int_B \int_{B'} \left( K(x,x)K(x',x') + 2K(x,x')^2 \right) - \sigma^4 \int_B \int_{B'} K(x,x)K(x',x') dxdx' \\
&= 2 \sigma^4 \int_B \int_{B'} K(x,x')^2 dx dx'
\end{align*}
which is (\ref{eq:variance-formula}).

If the fourth moment $\E[c^4]$ does not match that of a Gaussian, then the same method shows that the covariance is given by
\begin{align*}
{\rm cov}\left[\int_B |\phi|^2, \int_{B'} |\phi|^2 \right] = &2\sigma^4 \int_B \int_{B'} K(x,x')^2 dxdx' \\
&+ \big(\E[c^4] - 3\sigma^4 \big) \int_B \int_{B'} \sum_j \phi_j(x)^2 \phi_j(x')^2 dxdx'.
\end{align*}
\end{proof}
Note that, whereas $\sum_j \phi_j(x)\phi_j(x')$ is unaffected by an orthogonal change of basis $\phi_j \mapsto \sum_k a_{jk}\phi_k$, the sum of squares $\sum_j \phi_j(x)^2\phi_j(x')^2$ may depend on the choice of orthonormal basis. If $\E[c^4] = 3\sigma^4$, then this extra term disappears. 

\section*{Acknowledgments}

We thank Peter Sarnak for his advice, encouragement, and support over the course of this work. We thank Yaiza Canzani for helpful discussions about Weyl's law. We thank the Natural Sciences and Engineering Research Council of Canada for its support through a PGS D grant.


\begin{thebibliography}{99}

\bibitem{A} N.~Anantharaman, \emph{Entropy and the localization of eigenfunctions}, Annals of Math. (2), 168 (2008), 435–475. 

\bibitem{AN} N.~Anantharaman and S. Nonnenmacher, \emph{Half-delocalization of eigenfunctions for the Laplacian on an Anosov manifold}, Ann. Inst. Four. (Grenoble), 57, 6 (2007), 2465–2523. 

\bibitem{AS} N.~Anantharaman and L.~Silberman, \emph{A Haar component for quantum limits on locally symmetric spaces}, Israel J. Math. v195 no.1 493-447 (2013)

\bibitem{BouLi} J.~Bourgain and E.~Lindenstrauss, \emph{Entropy of quantum limits}, Comm. Math. Phys.,
233 (2003), 153–171. 

\bibitem{BL} N.~Burq and G.~Lebeau, \emph{Injections de Sobolev probabilistes et applications.}  Ann. Sci. \'{E}c. Norm. Sup\'{e}r. (4), 46 (2013), 917–962. arXiv:1111.7310. (2011)

\bibitem{CH} Y.~Canzani and B.~Hanin. \emph{High Frequency Eigenfunction Immersions and Supremum Norms of Random Waves.} Electronic Research Announcements in Mathematical Sciences, Volume 22, 2015, pp. 76-86. arXiv: 1406.2309.

\bibitem{CH2} Y~Canzani and B.~Hanin, \emph{Scaling limit for the kernel of the spectral projector and remainder estimates in the pointwise Weyl law}, Analysis \& PDE, Vol. 8, No. 7 (2015), 1707-1732

\bibitem{CdV} Y.~Colin de Verdi\`ere, \emph{Ergodicit\'{e} et les fonctions propres du laplacien} Comm. Math. Phys., 102 (1985), 497–502.

\bibitem{dci} M.~de Courcy-Ireland, \emph{Small-scale equidstribution for random spherical harmonics}, \arXiv{1711:01317}

\bibitem{DJ} S.~Dyatlov and L.~Jin. \emph{Semiclassical measures on hyperbolic surfaces have full support},
Acta Mathematica \textbf{220 }(2018) 297-339 arXiv:1705.05019

\bibitem{GW} A.~Granville and I.~Wigman, \emph{Planck-scale mass equidistribution of toral Laplace eigenfunctions}, Commun. Math. Phys. (2017) 355: 767. https://doi.org/10.1007/s00220-017-2953-3 arXiv:1612:07819.

\bibitem{Han} X.~Han, \emph{Small Scale Equidistribution of Random Eigenbases}, Commun. Math. Phys. 349, 425–440 (2017)

\bibitem{HT} X.~Han and M.~Tacy, \emph{Equidistribution of random waves on small balls}, preprint arXiv:1611.05983v2

\bibitem{HW} D.~L.~Hanson and F.~T.~Wright, \emph{A bound on tail probabilities for quadratic forms in independent random variables}, Ann. of Math. Stats., 1971, Vol. 42, No. 3, 1079-1083

\bibitem{H} R.~Holowinsky, \emph{Sieving for mass equidistribution}, Ann. of Math. (2), 172 (2010),
1499–1516. 

\bibitem{HS} R.~Holowinsky and K.~Soundararajan, \emph{Mass equidistribution of Hecke eigenfunctions}, Ann. of Math. (2), 172 (2010), 1517–1528. 

\bibitem{H} L.~H\"{o}rmander, \emph{The spectral function of an elliptic operator}, Acta Mathematica (1968)

\bibitem{Hum} P.~Humphries, \emph{Equidistribution in Shrinking Sets and $L^4$-Norm Bounds for Automorphic Forms}, Mathematische Annalen, 1-47 (2018). doi:10.1007/s00208-018-1677-9 arXiv:1705.05488

\bibitem{J} D.~Jakobson \emph{Quantum unique ergodicity for Eisenstein series on $PSL_2({\mathbb {Z}}\backslash PSL_2({\mathbb {R}})$.} Annales de l'institut Fourier 44.5 (1994): 1477-1504. http://eudml.org/doc/75106

\bibitem{L} P.~D. Lax, \emph{Asymptotic solutions of oscillatory initial value problems}, Duke Math. J. 24 1957, pp. 627-46

\bibitem{LR} S.~Lester and Z.~Rudnick, \emph{Small scale equidistribution of eigenfunctions on the torus} Commun. Math. Phys. 350 (2017), no. 1, 279-300

\bibitem{L1} E.~Lindenstrauss, \emph{Invariant measures and arithmetic quantum unique ergodicity}, Ann. of Math. (2), 163 (2006), 165–219. 

\bibitem{L2} E.~Lindenstrauss, \emph{On quantum unique ergodicity for $\Gamma \backslash \mathbb{H}\times \mathbb{H}$}, Internat. Math. Res. Notices 2001, 913–933. 

\bibitem{RS} Z.~Rudnick and P.~Sarnak, \emph{The Behaviour of Eigenstates of Arithmetic Hyperbolic Manifolds}, Commun. Math. Phys. 161, 195-213 (1994)

\bibitem{Sh1} A.~Shnirelman, \emph{Ergodic properties of eigenfunctions}, Uspenski Math. Nauk {\bf 29/6} (1974), 181–182.

\bibitem{Sh2} A.~Shnirelman, Appendix to \emph{KAM theory and semiclassical approximations to eigenfunctions} by V.~Lazutkin, Ergebnisse der Mathematik, 24, Springer-Verlag, Berlin, 1993.

\bibitem{VdK} J.M.~VanderKam, \emph{$L^{\infty}$ Norms and Quantum Ergodicity on the Sphere}, International Mathematics Research Notices 1997, no.7 p.329-47

\bibitem{Z} S.~Zelditch, \emph{Uniform distribution of eigenfunctions on compact hyperbolic surfaces}, Duke Math. J., 55 (1987), 919–941.

\end{thebibliography}
\end{document}